\documentclass[12pt]{amsart}
\usepackage{ulem}
\usepackage{cancel}
\usepackage{mathptmx, amsmath, amssymb, amsfonts, amsthm, mathptmx, enumerate, color,mathrsfs}
\usepackage{upgreek}
\usepackage{amsfonts}
\usepackage{txfonts}
\usepackage{cleveref}
\usepackage{graphicx}
\usepackage{xfrac}
\usepackage{amsfonts}
\usepackage{amsmath}
\usepackage{amssymb}
\usepackage{eurosym}
\usepackage{amsfonts}
\usepackage{amsmath}
\usepackage{enumerate}
\usepackage{color}
\usepackage{cite}
\usepackage{cleveref}
\usepackage{epstopdf}
\usepackage{dsfont}
\usepackage{url}

\newtheorem{theorem}{Theorem}[section]
\newtheorem{lemma}[theorem]{Lemma}

\theoremstyle{definition}

\newtheorem{remark}[theorem]{Remark}
\numberwithin{equation}{section} \textwidth=16cm \textheight=21cm
\begin{document}
\title{Localization of critical  points in annular conical sets via the method of Nehari manifold}
\author{Andrei Stan}
\address{
Department of Mathematics, Babeș-Bolyai
University AND Tiberiu Popoviciu
Institute of Numerical Analysis, Romanian Academy}
\maketitle

\section{Introduction and Preliminaries}

 The theory of critical points has proved to be a cornerstone in the study of various problems arising from real-world mathematical models. In many cases, the solutions to such problems  correspond to a critical point of some   $C^1$ energy functional \( E \). Specifically, one seeks \( u  \) such that 
\[
E'(u) = 0,
\]  
where $E'$ denotes the Fréchet derivative of $E$. 
In the upcoming discussion, \( H \) denotes a Hilbert space, identified with its dual, equipped with the inner product \((u, v)_H\) and the associated norm \(|u|_H = \sqrt{(u, u)_H}\). Also,  \( B_R \) denotes the closed ball of radius \( R \) centered at the origin in \( H \). 

\subsection{Nehari manifold}

A significant advancement in the study of critical points is attributed to the pioneering papers of Nehari (see \cite{nehari original1, nehari original2}). The core idea involves minimizing the functional \( E \) over the so-called Nehari   manifold,  
\begin{equation*}\label{nehari}
    \mathcal{M} = \left\{ u \in H \setminus \{ 0 \} \, : \, (E'(u), u)_H = 0 \right\},
\end{equation*} 
since in many cases the functional \( E \) is bounded on \( \mathcal{M} \), even though it may be unbounded on the entire domain. Moreover, by its  definition, all critical points of \( E \) are contained within \( \mathcal{M} \). 
Typically, the Nehari method proceeds as follows,\begin{itemize}
    \item[1)] Prove that the infimum of \( E \) over \( \mathcal{M} \) is achieved, i.e., there exists \( u_0 \in \mathcal{M} \) such that \( E(u_0) = c = \inf_{\mathcal{N}} E(\cdot) \).
    \item[2)] Establish that \( \mathcal{M} \) is homeomorphic to the unit sphere \( S^1 \) in \( H \) via a mapping \( \omega \). Then, show that \( u \) is a critical point of the mapping \( E \circ \omega \colon S^1 \to \mathbb{R} \). In the standard approach from  the literature, the mapping $\omega$ is obtained as the unique maximizer  of the fibering mapping $(0, \infty)\ni t\mapsto E(tu)$ (see, e.g., \cite{pohozaev}).
    \item[3)] Show  that any critical point of  \( E \circ \omega \) corresponds to a critical point of \( E \).
\end{itemize}
The point $u_0$ is usually called a ground state since it minimizes the functional over the set of all possible solutions. However, 
as shown in \cite[Theorem~4.2]{minimax}, $E(u_0)$ is just a  mountain pass value given by a min-max procedure over a specific set of paths.

In cases where \( \mathcal{M} \) is a \( C^1 \) manifold (e.g., when \( E \) is  a \( C^2 \) functional), the mapping \( \omega \) is of \( C^1 \) class, allowing us to work directly on the manifold \( \mathcal{M} \). For further details, we refer the reader to the remarkable monograph  by Szulkin and Weth \cite{nehari}. Additionally, some excellent sources include \cite{ambroseti,figueredo} and the references therein.

\subsection{Localization} 

In addition to establishing the existence of a critical point, one may also ask for its localization or, more generally, aim to identify a critical point within a specific subset. A localization result, particularly with respect to an energetic norm, may be of interest because, for example, when modeling a real process, the parameters of the nonlinear term can be adjusted to obtain a solution whose energy lies within predefined limits. Another motivation for localization arises in nonlinear problems, where the uniqueness of solutions may fail. In such cases, it is often important to focus the analysis on a specific solution, which requires localization.

In general, to achieve localization results for nonlinear problems, one typically employs a combination of inequalities, such as Poincaré's and Harnack's inequalities, alongside abstract localization techniques, e.g.,  Krasnoselskii-type methods or approaches that use the properties of topological degree.  

One of the earliest results on the localization of critical points can be found in Schechter \cite[Chapter 5.3]{schecter}, which provides sufficient conditions to ensure the existence of a minimizing sequence of a functional within a given ball. Under the assumption of the Palais-Smale condition, this minimizing sequence has a convergent subsequence whose limit is a critical point.
\begin{theorem}[Schechter's Theorem]  \label{Schechter}
   Assume \( E \) is bounded from below on \( B_R \) (\( R > 0 \)) and the following Leray-Schauder  boundary condition holds \begin{equation*}
       E'(u)+\mu u\neq 0\, \, \text{ for all }u\in \partial B_R \text{ and }\mu>0.
   \end{equation*}  Then, there exists a sequence \( \{u_n\} \subset B_R \) such that \begin{equation*}
       \text{\( E(u_n) \to \inf_{B_{R}} E \) \quad and  \quad $ E'(u_n) \to 0 $}.
   \end{equation*} 
\end{theorem}  
In \cite[Theorem 3.1]{p1}, a result analogous to Schecter's theorem was established for the  set $\overline{B_R\setminus B_r}$ ($0<r<R$), while in \cite[Theorem 2.3]{p2}, a similar result was obtained for the conical set
\[
K_{r,R} = \left\{ u \in K \, : \, r \leq |u|_1 \text{ and } |u|_2 \leq R \right\} \quad (0<r<R<\infty),  
\]  
where \( |\cdot|_i \) (\( i = 1, 2 \)) are two norms on \( H \), and \( K \) is a wedge, i.e., \( K \) is a closed convex subset of \( H \) satisfying \( \lambda K \subset K \) for all \( \lambda > 0 \). 
Further results on the localization of critical points can be found in \cite{p3, precup, precup2, galewski}.

\subsection{Ekeland variational principle}

The proof of our main result is essentially based on the weak form of
Ekeland's variational principle (see, e.g., \cite{ekeland,f}).
\begin{lemma}[Ekeland Principle - weak form]
\label{ekeland} Let $(X,d)$ be a complete metric space and let $\Phi
:X\rightarrow \mathbb{R} \cup \{+\infty \}$ be a lower semicontinuous and
bounded from below functional. Then, given any $\varepsilon >0$, there exists $%
u_{\varepsilon }\in X$ such that
\begin{equation*}
\Phi (u_{\varepsilon })\leq \inf_{X}\Phi +\varepsilon,
\end{equation*}
and
\begin{equation*}
\Phi (u_{\varepsilon })\leq \Phi (u)+\varepsilon d(u,u_{\varepsilon }),
\end{equation*}
for all $u\in X.$
\end{lemma}
\subsection{Implicit function theorem}

In our analysis, we employ the following variant of the implicit function theorem.  For another use of the implicit function theorem in the context of Nehari manifolds, see \cite[Proposition 4.2]{bahri}.
\begin{theorem}
    Let $A,B$ be open sets in $\mathbb{R}$, and let  \( x_0 \in \text{Int } A \), \( y_0 \in \text{Int } B \). Suppose that \( \mathcal{F}(y,x)\colon B\times A\to \mathbb{R} \) is a  function that satisfies the following conditions:
\begin{enumerate}
    \item \( \mathcal{F}( y_0,x_0) = 0 \);
    \item The function $\mathcal{F}(y,x)$ is continuously  differentiable with respect to both variables in a neighborhood of $(y_0,x_0)$;
    \item \( \mathcal{F}_y( y_0,x_0) \neq 0 \);
\end{enumerate}
Then:
\begin{enumerate}
    \item[(a)] There exist  neighborhoods \( U_0\subset A \) of \( x_0 \) and  \( V_0\subset B \) of \( y_0 \), as well as a unique mapping \( y = \xi(x) \colon U_0 \to V_0 \) such that:
    \[
    \xi(x_0) = y_0 \quad \text{and} \quad \mathcal{F}( \xi(x),x) = 0 \quad \text{for all } x \in U_0;
    \]
    \item[(b)] The mapping  $\xi$ is continuously differentiable  on \( U_0 \) and moreover satisfies
    \[
    \xi'(x) = -\frac{\mathcal{F}_{x}(\xi(x),x)}{\mathcal{F}_y( \xi(x),x)} \, \, \text{ for all }x\in U_0.
    \]
\end{enumerate}
\end{theorem}


\section{Main Result}  

  Let \( E : H \to \mathbb{R} \) be a twice Fréchet differentiable functional, and let $K$ be a nondegenerate cone in $H$, i.e., \begin{equation*}
      \text{$K\setminus\left\{0\right\}\neq \emptyset$, $\mathbb{R}_+ K\subset K$ and $K+K\subset K$.}
  \end{equation*} The second derivative of $E$  at $x$ in the direction $y$ on the point $z$ is  (see, e.g., \cite{gateaux}):
    \begin{equation*}
        E''(x)(z,y)=\lim_{t\searrow 0}\frac{1}{t} \left( E'(x+ty)-E'(x),z\right)_H.
    \end{equation*}  
Throughout this paper, we assume that the operator $N\colon H\to H$ given by   \begin{equation*}\label{N}
    N(u)=u-E'(u) \, \, \text{ for all }u\in H,
\end{equation*} 
    is invariant over $K$, i.e., $N(K)\subset K$ (recall that $H$ is identified with its duals, so $E': H \to H$).

Our aim is to determine a  critical point of $E$ within the conical set 
\begin{equation*}
   \label{conical set}
   K_{r,R}=\left\{
   u\in K\setminus\{0\}\,: \, r\leq |u|_H\leq R 
   \right\},
\end{equation*}
where $0<r<R<\infty$ are some given real numbers. 
The main assumption we consider on the functional $E$ is the following:
\begin{description}  
\item[\textbf{(h1):}] For each \( u \in K \setminus \{ 0 \} \), there exists a unique \( s(u) \in \left( \frac{r}{|u|_H}, \frac{R}{|u|_H} \right) \) such that the mapping  
\[
\tau \mapsto \left( E'( \tau u), u \right)_H \quad (\tau > 0),
\]  
is strictly positive on \( \left[ \frac{r}{|u|_H}, s(u) \right) \) and strictly negative on \( \left( s(u), \frac{R}{|u|_H} \right] \).  
\end{description}  
From this, it follows that the mapping \(\alpha_u(\tau) = E(\tau u)\), for some \(u \in K\setminus\{0\}\), has a unique critical point within the interval \(\left( \frac{r}{|u|_H}, \frac{R}{|u|_H} \right)\) at \(\tau = s(u)\), and is concave at this point. Thus, given the smoothness of the functional \(E\), these properties are characterized by
\begin{equation*}\label{caracterizare puncte de pe varietate}
\alpha_u^{\prime}(s(u)) = 0 \quad \text{and} \quad \alpha_u^{\prime\prime}(s(u)) \leq 0.    
\end{equation*}
Moreover, we see that $s(ku)=\frac{s(u)}{k}$ for all  $k>0$.

Following the method of Nehari manifold, we look for critical points of $E$ on $K_{r,R}$ within the conical shell
\begin{equation*}
    \Tilde{\mathcal{N}}=\left\{ u \in K_{r,R}\, : \, \left(E'(u),u \right)_H=0\right\}.
\end{equation*}
Further, let $\mathcal{N}$ be the set
\begin{equation*}
   \mathcal{N}= \left\{ s(u)u \, : \, u\in K\setminus \{0\}\right\},
\end{equation*}
and observe that $\mathcal{N}= \Tilde{\mathcal{N}}$. To see this, from
 (h1) one clearly has  $ \mathcal{N}\subset \Tilde{\mathcal{N}}.$ Conversely,  if  $(E'(u),u)_H=0$ for some $u\in K\setminus\{0\}$, then $s(u)=1$, which implies $u\in \mathcal{N}$.

The main result of this paper is presented in \Cref{th principala} below, where we establish an analogue of the results presented in \cite{p2}, obtained using  the method of Nehari manifold.  \begin{remark}
    We mention that in \cite{figueredo}, the existence of a ground state for a $C^1 $ functional was established over an open cone in Banach spaces. However, our result do not require this condition, and in applications, the cones often have empty interior (as the one from Section 3).
\end{remark}
\begin{theorem}\label{th principala}  
Assume condition (h1) holds. In addition, we suppose that
\begin{description}
\item[\textbf{(h2):}] The functional \( E \) is bounded from below on $\mathcal{N}$. 
\item[\textbf{(h3):}] The second Fréchet derivative $E'' (u)$ is bounded uniformly with respect to $u\in \mathcal{N}$, i.e., there exists $C_1>0$ such that
\[
\sup_{\substack{w_1, w_2 \in H  \\ |w_1|_H = |w_2|_H = 1}} \left|E''(u)(w_1, w_2) \right|\leq C_1 \, \, \text{ for all }u\in \mathcal{N}.
\]
\item[\textbf{(h4):}]  There is a positive constant $C_2>0$ satisfying \begin{equation*}
    |E''(u)(u,u)|\geq C_2>0 \, \, \text{ for all }u\in \mathcal{N}.
\end{equation*}
\end{description}
Then, there exists a sequence \( \{u_n\} \subset \mathcal{N} \) such that
\[
E(u_n) \to \inf _{\mathcal{N}} E \quad \text{ and }\quad E'(u_n) \to 0.
\]  
\end{theorem}  
The following auxiliary result will be of great importance in proving Theorem \ref{th principala}.
\begin{lemma}  \label{lema_principala}
    Let $u\in \mathcal{N}$, $v\in H$ and let $\varepsilon>0$ be such that  $$u+tv\in K\, \, \text{ for  all $t\in [0,\varepsilon]$. }$$ If assumption (h4) holds true, then the limit \begin{equation*}
        \lim_{t\searrow 0} \frac{s(u+tv)-s(u)}{t}
    \end{equation*} 
    exists and has the value \begin{equation*}
        -\frac{E'' (u)(u,v)+(E'(u),v)_H}{E'' (u)(u,u)}.
    \end{equation*}
\end{lemma}
\begin{proof}
Let us consider the mapping $g\colon (0, \infty)\times (-\varepsilon,\varepsilon) \to \mathbb{R}$ given by 
\[
g(\tau, t) = \left(E'(\tau(u + tv)), u + tv \right)_H.
\]   
Note that  $g$ is continuously differentiable, with partial derivatives 
    \begin{align}\label{derivata g 1}
        g'_\tau(\tau,t)&=\lim_{\delta\searrow 0}\frac{1}{\delta}\left(E'(\tau(u+tv)+\delta(u+tv))-E'(\tau(u+tv)), u+tv \right)_H
        \\&
        =E'' (\tau(u+tv))(u+tv,u+tv),\nonumber
    \end{align}
    and 
    \begin{align}\label{derivata g 2}
                g'_t(\tau,t)&=\lim_{\delta\searrow 0}\left(\frac{1}{\delta}\left(E'(\tau(u+tv)+\delta\tau v), u+tv+\delta v \right)_H-\left(E'(\tau(u+tv)), u+tv \right)_H\right)
                \\&
                =\left(E'(\tau(u+tv),v\right)_H+\lim_{\delta\searrow 0}\frac{1}{\delta}\left(E'(\tau(u+tv)+\delta
\tau v)-E'(\tau(u+tv)), u+tv \right)_H \nonumber
                \\&
=\left(E'(\tau(u+tv),v\right)_H+E''(\tau(u+tv))\nonumber (u+tv,\tau v).
    \end{align}
Additionally, we have
\begin{equation*}\label{conditii fct implicita}
g(1, 0) = 0 \quad \text{ and }\quad g'_\tau(1, 0)< 0.
\end{equation*}
Indeed,  the first relation  follows immediately since $s(u)=1$, while the second one follows from \eqref{derivata g 1} and (h4).

Employing the 
implicit function theorem, there exists $\varepsilon_0>0$ and a unique continuously differentiable mapping \(\xi \colon (-\varepsilon_0, \varepsilon_0) \to \mathbb{R}\) such that  $\xi(0) = 1$,
\[
 g(\xi(t),t)=0\quad \text{and} \quad \xi'(t) = -\frac{g'_t(\xi(t), t)}{g'_\tau(\xi(t), t)}\, \, \text{ for all }t\in (-\varepsilon_0, \varepsilon_0).
\]  

On the other hand, for each $t\in [0, \varepsilon]$, since \(u + tv \in K\), assumption \((h1)\) guarantees the existence of a unique value \(s(u+tv)\) within the interval \(\left(\frac{r}{|u+tv|_H}, \frac{R}{|u+tv|_H}\right)\) satisfying 
\[
(E'(s(u+tv)(u + tv)), u + tv)_H= 0,\quad \text{ i.e., } \quad g(s(u+tv),t)=0.
\] 
Since $\xi$ is smooth and $\xi(0)=1\in \left( \frac{r}{|u|_H},\frac{R}{|u|_H}\right)$, there exists $\varepsilon_1>0$ such that $$\xi(t)\in \left(\frac{r}{|u+tv|_H}, \frac{R}{|u+tv|_H}\right) \, \, \text{  for all $t\in (-\varepsilon_1,\varepsilon_1)$. }$$
Consequently,  we have
\begin{equation*}
    \xi(t)=s(u+tv) \, \, \text{ for all }t\in [0,\varepsilon_2),
\end{equation*}
where $\varepsilon_2=\min\{\varepsilon,\varepsilon_0,\varepsilon_1\}$.  
Whence 
 \begin{equation*}
        \lim_{t\searrow 0} \frac{s(u+tv)-s(u)}{t}=\xi'(0)=-\frac{g'_t(1, 0)}{g'_\tau(1, 0)}.
    \end{equation*} 
    Now, taking $\tau=1$ and $t=0$ in \eqref{derivata g 1} and \eqref{derivata g 2}, we derive the conclusion
     \begin{equation*}
        \lim_{t\searrow 0} \frac{s(u+tv)-s(u)}{t}=-\frac{E'' (u)(u,v)+(E'(u),v)_H}{E'' (u)(u,u)}.
    \end{equation*}
     
\end{proof}

\begin{proof}[Proof of Theorem \ref{th principala}]
From assumption (h2),  $E$ is bounded from below on $\mathcal{N}$. Moreover, since $\mathcal{N}$ is closed, we may apply Ekeland's variational principle to  the functional $E$ on the set $\mathcal{N}$.  This guarantees the existence of a sequence $\{u_n\}\subset  \mathcal{N}$ such that 
\begin{align}
E(u_n) &\leq \inf_{\mathcal{N}} E + \frac{1}{n}, \\
E(u_n) &\leq E(u) + \frac{1}{n} |u_n - u|_H\, \, \text{for all } u \in \mathcal{N}.\label{ekeland2}
\end{align} 
For each $n\in \mathbb{N}$, consider the mapping \(\varphi \colon \mathbb{R} \to H\) given by 
\[
\varphi(t) = u_n - tE'(u_n).
\]  
Clearly, $\varphi$ is continuously differentiable  with $\varphi'(0)=-E'(u_n)$. Moreover, we have 
 \(\varphi(t) \in K\) for all  \(t \in [0, 1]\). Indeed, since both $u_n\in K$ and \(N(u_n) \in K\), one has
\[
u_n - tE'(u_n) = (1-t)u_n + t\left(u_n-E'(u_n) \right)=(1-t)u_n+t N(u_n)\in K.
\]
Next,  we define 
\begin{equation*}
    \psi(t)=s(\varphi(t))\varphi(t)\in \mathcal{N} \, \,\text{ for all }t\in [0,1].
\end{equation*}
Choosing 
 \(u = \psi(t)\)  in \eqref{ekeland2} yields
\begin{equation}\label{ekeland3}
  E(\psi(0)) \leq E(\psi(t)) + \frac{1}{n}|\psi(t)-\psi(0)|_H\, \, \text{ for all }t\in [0,1].  
\end{equation}
By Lemma \ref{lema_principala},  the mapping $t \mapsto s(\varphi(t))$ has a right derivative at zero given by \begin{equation*}
    z_n= \frac{E''(u_n)(u_n, -E'(u_n)) - |E'(u_n)|_H^2}{E''(u_n)(u_n, u_n)}.
\end{equation*}
Thus, since $\varphi(0)=u_n$, when differentiating, we obtain
\begin{equation}\label{derivata psi}
    \psi'_+(0)=\lim_{t\searrow0} \frac{\psi(t)-\psi(0)}{t}=z_n \,u_n-E'(u_n).
\end{equation}
Regarding $z_n$, from assumption (h3) and (h4), one has
\begin{equation}\label{estimare z_n}
    |z_n|\leq \frac{C_1}{C_2} |u_n|_H |E'(u_n)|_H+\frac{1}{C_2}|E'(u_n)|^2 _H.
\end{equation}
Now, dividing  relation \eqref{ekeland3} by $t>0$ gives
\begin{align}\label{ekeland4}
    \frac{E(\psi(0))-E(\psi(t))}{t}\leq \frac{1}{n} \left|\frac{\psi(t)-\psi(0)}{t} \right|_H.
\end{align}
Taking the limit as $t\searrow 0$, from \eqref{derivata psi} and $\left( E'(u_n),u_n\right)_H=0$, the left-hand side of \eqref{ekeland4} becomes
\begin{align}\label{stanga}
    \lim_{t\searrow 0}\frac{E(\psi(0))-E(\psi(t))}{t}&=-\left(E'(\psi(0)),\psi'_+(0) \right)_H
    \\& =-\left(E'(u_n),z_n u_n-E'(u_n) \right)_H \nonumber
    \\&=|E'(u_n)|^2_H.\nonumber
\end{align}
For the right-hand side of \eqref{ekeland4}, using \eqref{derivata psi} and \eqref{estimare z_n}, we obtain
\begin{align}\label{dreapta}
   \lim_{t\searrow 0}   \frac{1}{n} \left|\frac{\psi(t)-\psi(0)}{t} \right|_H \leq \frac{C}{n}\left( |u_n|_H |E'(u_n)|_H+|E'(u_n)|^2 _H\right).
\end{align}
Consequently, from \eqref{ekeland4}, \eqref{stanga} and \eqref{dreapta}, we obtain 
\begin{equation*}
    |E'(u_n)|_H\leq \frac{C}{n}\left(|u_n|_H+|E'(u_n)| _H \right)\leq  \frac{C}{n}\left(R+|E'(u_n)| _H \right),
\end{equation*}
which ensures that $E'(u_n)\to 0$ as $n\to \infty$.

\end{proof}

If we further assume a compactness condition, Theorem \ref{th principala} leads to the following critical point principle.

\begin{theorem} Assume (h1)-(h4) hold true. If in addition the operator $N$ given in \eqref{N} is completely continuous, then there exists $u^\ast\in \mathcal{N}$ such that \begin{equation*}
  E(u^\ast)=\inf_\mathcal{N}E\quad \text{ and } \quad E'(u^\ast)=0.
\end{equation*}\label{corolar}
\end{theorem}
\begin{proof}
    From assumptions (h1)-(h4), Theorem \ref{th principala} guarantees the existence of a sequence $\{u_n\}\subset \mathcal{N}$ such that \begin{equation}
     E(u_n)\leq \inf_{\mathcal{N}}E +\frac{1}{n}\quad \text{ and } \quad   E'(u_n)\to 0.\label{PS sequence}
    \end{equation}
    Since $N$ is completely continuous and $\{u_n\}$ is bounded (recall that $\mathcal{N}\subset K_{r,R}$), it follows that, after possibly passing to a subsequence, the sequence $\{N(u_n)\}$ is convergent to some $u^\ast\in H$. Now, taking the limit in \eqref{PS sequence}, one has that $u_n\to u^\ast$, hence $u^\ast\in \mathcal{N}$, $E(u^\ast)=\inf_{\mathcal{N}}E$ and $E'(u^\ast)=0$.
\end{proof}
Finally, we note that Theorem \ref{corolar} immediately yields multiplicity results if the hypotheses are satisfied for several finite or infinitely many pairs of numbers.
\begin{remark}[Multiplicity]
Let \(\{r_i\}_{1 \leq i \leq m}\) and \(\{R_i\}_{1 \leq i \leq m}\) be two sequences of positive real numbers satisfying 
\[
0 < r_1 < R_1 < r_2 < R_2 < \dots < r_m < R_m < \infty.
\]
For each pair \((r_i, R_i)\), we assume that condition \((h1)\) holds with \(s_i := s\). Additionally, suppose that conditions \((h2)\)–\((h4)\) hold for each Nehari manifold
\[
\mathcal{N}_i = \left\{ u \in K \setminus \{0\} \, : \, s_i(u)u \right\}, \quad i = 1, \dots, m,
\]
and that the operator \(N\) is completely continuous. Then, for each \(i = 1, \dots, m\), there exists  \(u^\ast_i \in \mathcal{N}_i\)  such that
\[
E(u^\ast_i) = \inf_{u \in \mathcal{N}_i} E(u) \quad \text{and} \quad E'(u^\ast_i) = 0.
\]

\end{remark}

\section{Application}

In this section, we present an application of Theorem \ref{corolar}. We aim to obtain a symmetric and positive solution for the Dirichlet problem
\begin{equation}\label{ec aplicatie}
    \begin{cases}
        -u^{\prime \prime }\left( t\right) =g(t)f\left( u\left( t\right) \right) \, \,
\text{ on }\left[ 0,1\right], \\ 
u\left( 0\right) =u\left( 1\right) =0,
    \end{cases}
\end{equation} 
where $f\colon \mathbb{R} \to \mathbb{R}$ is  continuously differentiable. In addition,  we suppose that $f$ is   nondecreasing and positive on $\mathbb{R}_+$. The function $g\colon [0,1]\to \mathbb{R}_+$ is assumed to be bounded,  nondecreasing on $[0,1/2]$ and symmetric with respect to the middle of the interval, i.e., 
\begin{equation*}
    g(t)=g(1-t) \quad \text{ for all }t\in [0,1/2].
\end{equation*}  Note that, since $g$ is nondecreasing and bounded on $[0,1/2]$, it is measurable on this interval, and by symmetry,  the same holds on the entire interval $[0,1]$.\newline
Consider the Sobolev space $H=H_0^1 (0,1)$ endowed with the inner product \begin{equation*}
    (u,v)_{H_0^1}=\int_0^1 u'(t)v'(t)dt,
\end{equation*} and the energetic norm $$|u|_{H_0^1}=\sqrt{\left( u,u\right)_{H_0^1}}. $$
We identify the space $H_0^1(0,1)$ with its dual $H^{-1}(0,1)$   via the mapping \begin{equation*}
    J\colon H_0^1 (0,1) \to  H^{-1}(0,1),\quad Ju=-u''.
\end{equation*}
Clearly, $J$ is invertible and its inverse $J^{-1}v$  is the weak solution of the Dirichlet problem $-u''(t)=v, \, u(0)=u(1)=0$. If $v\in L^2(0,1)$, then \begin{equation*}
    (J ^{-1}v)(t)=\int_0^1 G(t,s)v(s)ds,
\end{equation*}
where $G$ is the Green's function  of the differential operator $J$ with respect to
the boundary conditions $u(0)=u(1)=0$ (see, e.g., A. Cabada \cite[Example~1.8.18]{green}), $$    G(t,s)=\begin{cases}
     s\left( 1-t\right), s\leq t \\
       t\left( 1-s\right), s\geq t.
    \end{cases}$$
Moreover, given the continuous embedding of $H_0 ^1(0,1)$ in $C [0,1]$,  one has
\begin{equation}\label{scufundare}
   \sup_{t\in [0,1]} |u(t)|\leq |u|_{H_0^1}\, \, \text{ for all }u\in H_0^1(0,1).
\end{equation}
Additionally, the Wirtinger inequality holds \cite{brezis},
\begin{equation}\label{Wirtinger}
    \int_0 ^1 u^2(t)dt\leq \frac{1}{\pi^2} |u|_{H_0 ^1}^2.
\end{equation}
The energy functional of the problem \eqref{ec aplicatie} is $E\colon H_0^1(0,1)\to \mathbb{R},$
\begin{equation*}
E\left( u\right) =\frac{1}{2}\left\vert u\right\vert
_{H_{0}^{1}}^{2}-\int_{0}^{1}F\left( u(t)\right)g(t)  dt,
\end{equation*}
where 
$F\left( \xi\right)  =\int_{0}^{\xi }f\left(
s\right) ds.$
Clearly, the smoothness of $f$ implies that $E$ is a $ C^2$ functional.  Its first derivative is given by  $$    E'(u)=u-N(u)\quad \left( u\in H_0^1(0,1)\right), $$
where \begin{equation*}
     N(u)(t)=\int_0^1 G(t,s)  f(u(s)) g(s) ds \quad (u\in H_0^1(0,1)),
 \end{equation*}
while the  second derivative is expressed by \begin{equation*}
     E''(u)(w_1,w_2)=\left( w_1, w_2\right)_{H_0^1}-\int_0^1 f'(u(t)) 
 g(t) w_1(t)w_2(t)dt,
 \end{equation*}
 for  $u,w_1,w_2\in H_0^1(0,1).$

In $H_0 ^1 (0,1)$, we consider the cone
\begin{equation*}
    \begin{aligned}
K = \Big\{ u \in H_0^1(0,1) \, : \, & u \text{ is nondecreasing on } [0,1/2], \\
& u(t) = u(1-t)\text{ and } u(t) \geq \phi(t) |u|_{H_0^1} \text{ for all } t \in [0,1/2] \Big\},
\end{aligned}
\end{equation*}
where \begin{equation*}
   \phi \colon [0,1/2] \to \mathbb{R}, \text{ \quad\(\phi(t) = t(1 - 2t)\).}
\end{equation*} Note that the function $u(t)=\sin(\pi t)$ belongs to $K$, so the cone $K$ is nondegenerate.

We claim that operator $N$ is invariant with respect to $K$, i.e., $$N(K)\subset K.$$
An important result in proving our claim is the 
following Harnack-type inequality obtained in \cite[Lemma 3.1]{precup}.
\begin{lemma}\label{lema Harnack}
   For every function \( u \in K \), with \( Ju \in C([0,1]; \mathbb{R}^+) \) nondecreasing on \([0,1/2]\), one has
\begin{equation}\label{inegalitate Harnack}
u(t) \geq \phi(t) |u|_{H_0 ^1} \quad \text{for all } t \in [0,1/2].
\end{equation}
\end{lemma}


Let $u\in K$.  From the definition of $N$, one has $$J(N(u))=g(\cdot)f(u)\in C[0,1].$$ Simple computations show that the Green's function satisfies \begin{equation*}
    G(t,s)=G(1-t,1-s) \quad\text{ and }\quad G(t,1-s)=G(1-t,s), \, 
\end{equation*} for all $t,s\in [0,1/2]$. Thus, from the symmetry of $u$, for any $t\in [0,1/2]$, we obtain
\begin{align*}
    N(u)(t)&=\int_0^{1/2}G(t,s)f(u(1-s)) g(1-s) ds+\int_0^{1/2}G(t,1-s)f(u(s)) 
 g(s)ds\\&=
    \int_{1/2}^{1}G(1-t,s)f(u(s)) g(s)ds+\int_0^{1/2}G(1-t,s)f(u(s)) g(s)ds\\&
    =N(u(1-t)),
\end{align*}
which proves that $N(u)$ is  symmetric.
 Since  $u$ takes positive values and both $f$ and $g$ are nonnegative, it follows that $N(u)$ is concave, hence \(N(u)\) is nondecreasing on \([0,1/2]\). Furthermore, the monotonicity of \(f\) on \(\mathbb{R}_+\) and of $u$ and $g$ on $[0,1/2]$ ensures that \(J(N(u))=g(\cdot)f(u)\) is also nondecreasing on \([0,1/2]\).  Therefore, by Lemma \ref{lema Harnack},  inequality \eqref{inegalitate Harnack} holds for \(N(u)\). Thus, \(N(u) \in K\), as desired.

For the sake of completeness, below we provide a  proof of Lemma \ref{lema Harnack} shorter than the one given in \cite{precup}.
\begin{proof}[Proof of Lemma \ref{lema Harnack}]
 From the hypothesis, it follows directly that both \( u \) and \( u' \) are positive and concave on \( [0,1/2] \). Moreover, the symmetry of \( u \) ensures that \( u \) is increasing on \( [0,1/2] \), while \( u' \) is decreasing, with \( u'(1/2) = 0 \).  For any $t\in (0,1/2)$, the monotonicity of $u'$ implies
 \begin{equation}\label{Harnack1}
     u(t)=\int_0^t u'(s)ds\geq tu'(t),
 \end{equation}
 while its concavity yields
 \begin{align}\label{Harnack2}
     u'(t)&=u'\left(2t\,\frac{1}{2} +(1-2t)0 \right)\geq 2t u'\left(\frac{1}{2}\right)+(1-2t)u' (0)\\&=(1-2t)u' (0).
\nonumber \end{align}
 Finally, we obtain the conclusion using  \eqref{Harnack1}, \eqref{Harnack2} and 
 \begin{equation*}
     |u|_{H_0^1 }^2 =\int_0^1 u'(s)^2ds=2\int_0^{1/2} u'(s)^2ds\leq u'(0)^2.
 \end{equation*}
\end{proof}
Let $0<r<R<\infty$ be positive real numbers,  $\beta\in (0,1/4)$, and define
\begin{equation*}
    \Tilde{A}=\left(\int_0 ^1 g^2(t)dt\right)^{1/2},\quad \Tilde{B}= \int_0^\beta g(t)dt \quad\text{ and } \quad\Tilde{C}=\int_{\beta}^{1/2} g(t)dt.
\end{equation*}
In what follows, we assume
\begin{enumerate}
    \item[\textbf{(H1):}] The constants $\Tilde{A}$ and $\Tilde{C}$ are strictly positive, and moreover
    \begin{equation*}
        \frac{f(r)}{r} < \frac{\pi}{\Tilde{A}} \quad \text{and} \quad \frac{f\left( R \phi(\beta) \right)}{R} > \frac{1}{2\phi(\beta)\Tilde{C}}.
    \end{equation*} 
\end{enumerate}
Additionally, suppose that the function 
$f$ satisfies one of the following three  conditions:
\begin{enumerate}
\item[\textbf{(H2):}] There exists a continuous  mapping $\theta\colon [0,R]\to\mathbb{R}$ such that $\theta(t)>0$ for $t\in (0,R]$, and \begin{equation*}
    tf'(t)-f(t)\geq \theta(t) \, \,\text{ for all }t \in [0,R];
\end{equation*}  \end{enumerate}

\begin{enumerate}
    \item[\textbf{(H3):}] There exists constants $\mu=\mu(r,R)>1$ and $\lambda=\lambda(r,R)>0$ such that 
    \begin{equation*}\label{AR}
        tf'(t)-\mu f(t)\geq 0 \text{ and } f'(t)\geq \lambda \, \, \text{ for all }t\in [r\phi(\beta),R],
    \end{equation*}
   and
     \begin{equation}\label{estimare mu,lambda}
    \Tilde{B}f(r\phi(\beta))<\lambda\Tilde{C}\left( 1-\frac{1}{\mu}\right)r \phi(\beta);
    \end{equation}
\end{enumerate}
 \begin{enumerate}
    \item[\textbf{(H4):}] The  support of the function \( g \) is included in \( [\beta,1/2] \), i.e.,  
    \begin{equation}\label{g=0}
        g(t) = 0 \quad \text{ for all } t \in [0,\beta].
    \end{equation}  
    Additionally, the function \( f \) is of class \( C^2 \) on \( [r\phi(\beta), R] \) with a strictly positive second derivative, that is, there exists \( M = M(r,R) > 0 \) with  
    \begin{equation*}
        f''(t) \geq M \quad \text{ for all } t \in [r\phi(\beta), R].
    \end{equation*}  
    Furthermore, there exists a continuous positive mapping \( \Tilde{\theta} \colon [r\phi(\beta), R] \to \mathbb{R} \) such that  
    \begin{equation}\label{AR coroana}
        tf'(t) - f(t) \geq \Tilde{\theta}(t) \quad \text{ for all } t \in [r\phi(\beta), R].
    \end{equation}  
\end{enumerate}
We show that under condition (H1) and either (H2), (H3) or (H4), conditions (h1)-(h4) are satisfied.
Note that (h2) and (h3) follow directly from \eqref{scufundare}, Wirtinger's inequality, and the smoothness of the function $f$.  

To prove (h1),
let $u\in K\setminus\{0\}$.  The mapping $\alpha_u$ introduced in  (h1) is given by
\begin{equation*}
    \alpha_u (\tau)=\frac{1}{2}\tau^2|u|_{H_0 ^1}^2-\int_0^1F(\tau u(t)) g(t) dt, \quad \tau\in \left(\frac{r}{|u|_{H_0 ^1}},\frac{R}{|u|_{H_0 ^1}} \right)
\end{equation*}
and its derivative is \begin{align*}
     \alpha_u '(\tau)&=\tau |u|_{H_0 ^1 }^2-\int_0^1 f(\tau u(t)) g(t) u(t)dt.
\end{align*}
First, we show that \begin{equation}\label{capete}
    \text{$\alpha_u'\left(\frac{r}{|u|_{H_0 ^1}} \right)>0$ \quad and \quad $\alpha_u'\left(\frac{R}{|u|_{H_0 ^1}} \right)<0$}.
\end{equation}  Denote $$v(t)=\frac{u(t)}{|u|_{H _0 ^1}}\in K\setminus \{0\},$$ so  $|v|_{H_0 ^1}=1$. Then,  we have 
\begin{align*}
    \alpha_u'\left(\frac{\tau}{|u|_{H_0 ^1}} \right)=\tau|u|_{H _0 ^1} \left(1-\int_0^1 \frac{1}{\tau }f\left( \tau v(t)\right)g(t)v(t)dt\right).
\end{align*}
  From \eqref{scufundare},  (H1) and both Wirtinger and Hölder inequalities, one has
\begin{align*}
 &   \int_0^1 \frac{1}{r}f\left( rv(t)\right)  g(t) v(t)dt\leq \int_0^1\frac{f\left( r\right)}{r}  g(t)   v(t)dt\\&
    \leq \frac{f(r)}{  r}\left(\int_0 ^1 g^2(t)dt\right)^{1/2}\left(\int_0 ^1 v^2(t)dt\right)^{1/2}\\&
    \leq \frac{f(r)}{  r}\frac{ \Tilde{A}}{\pi} |v|_{H_0 ^1}=\frac{f(r)}{  r}\frac{ \Tilde{A}}{\pi}<1,
\end{align*}
which proves the first inequality in \eqref{capete}. For the second one, using the  monotonicity of $v$ and the Harnack inequality, we see that
\begin{equation}\label{inegalitate phi}
    v(t)\geq \phi(\beta)\quad \text{ for all }t\in [\beta,1/2].  
\end{equation}
Therefore, from (H1), the symmetry of $v$ and \eqref{inegalitate phi}, we obtain
 \begin{align*}
     \int_0^1 \frac{f\left( R v(t)\right)}{R} g(t) v(t)dt&= 2\int_0 ^{1/2} 
 \frac{f(Rv(t))}{R}  g(t) v(t)dt 
 \geq 2\int_{\beta}^{1/2} \frac{f(Rv(t))}{R}  g(t)v(t)dt \\& \geq  2\int_{\beta}^{1/2} \frac{f(R\phi(t))}{R}  g(t)\phi(t)dt 
 \geq 2  \phi(\beta) \Tilde{C} \frac{f\left(R\phi(\beta)\right)}{R}>1,
 \end{align*}
whence relation \eqref{capete} holds. 

Define $\sigma=\tau |u|_{H_0 ^1}$ and the functions $h,\Tilde{h}\colon [r,R]\to \mathbb{R}$,  \begin{align*}
   & h(\sigma  )=1-\int_0^1\frac{f(\sigma  v(t))}{\sigma  } g(t)  v(t)dt,\\&  \Tilde{h}(\sigma)=\sigma -\int_0^1f(\sigma  v(t)) g(t)  v(t)dt.
\end{align*} Clearly, $h$ and $\Tilde{h}$ are dependent on the chosen $u$, so  one should write $h_u$ and $\Tilde{h}_u$, respectively. However, for simplicity, we omit the subscript in the following analysis, as the dependence on \( u \) is clear from the context.

We now show that the function $h$ is decreasing under either (H2) or (H3), while the function $\Tilde{h}$ is concave  under (H4).

\textit{(a)}
In  case (H2) holds, the mapping 
\begin{equation*}
    t\mapsto \frac{f(t)}{t}\, \text{ is increasing on }(0,R].
\end{equation*}
The conclusion now follows directly,  as    $\sigma v(t)\in [0,R]$ for all $t\in [0,1]$ and  $\sigma\in [r,R]$.
 
\textit{(b)} 
Under condition (H3),  we prove that $h'(\sigma)<0$ for all $\sigma\in [r,R]$. Using the symmetry of $v$, a straightforward computation yields
\begin{equation*}
    h'(\sigma )=\frac{2}{\sigma^2}\int_0 ^{1/2}\left(  f(\sigma  v(t))-f'(\sigma  v(t)) \sigma  v(t)\right) g(t) v(t)dt.
\end{equation*} Let $\mathcal{B}$ denote the set
 \begin{equation*}
    \mathcal{B}=\{t\in [0,1/2]\, :\, \sigma  v(t)\leq r\phi(\beta)\}.
\end{equation*}
 From \eqref{inegalitate phi}, it follows that \begin{equation}\label{incluziune A}
    \mathcal{B}\subset [0,\beta] \quad\text{ and } \quad [\beta,1/2]\subset [0,1/2]\setminus  \mathcal{B}.
\end{equation} 
Since the derivative of $f$ is nonnegative, one has
\begin{align*}
    \frac{2}{\sigma^2}\int_{ \mathcal{B}} \left(  f(\sigma  v(t))-f'(\sigma  v(t)) \sigma  v(t)\right) g(t) v(t)dt &\leq  \frac{2}{\sigma^2}\int_ {\mathcal{B}}f(\sigma  v(t))g(t)v(t)dt.
\end{align*}
Now, using  the monotonicity of  $v$ and $f$, along with  \eqref{inegalitate phi} and \eqref{incluziune A}, we obtain
\begin{equation}\label{ineq1H3}
         \frac{2}{\sigma^2}\int_ {\mathcal{B}}f(\sigma  v(t))g(t)v(t)dt= \frac{2}{\sigma^3}\int_ {\mathcal{B}}f(\sigma  v(t))g(t)\sigma v(t)dt
    \leq \frac{2\Tilde{B}}{\sigma ^3} f(r\phi(\beta))r\phi(\beta).
\end{equation}
On the other hand, from (H3), since $\sigma v(t)\in [r\phi(\beta), R]$ for all $t\in [0,1/2]\setminus \mathcal{B}$, we derive 
\begin{align}\label{ineq2H3}
   & \frac{2}{\sigma ^2}\int_{[0,1/2]\setminus \mathcal{B}}\left( f'(\sigma v(t))\sigma v(t)-f(\sigma v(t)) \right)g(t)v(t)dt \\& \nonumber\geq 
   \frac{2}{\sigma ^3}\left( 1-\frac{1}{\mu}\right)\int_{[0,1/2]\setminus \mathcal{B}} g(t)f'(\sigma v(t))\sigma^2 v(t)^2dt\\&  \nonumber \geq \frac{2}{\sigma ^3}\left( 1-\frac{1}{\mu}\right)\int_{\beta}^{1/2-\beta}g(t) f'(\sigma v(t))\sigma ^2v(t)^2dt\\& \geq \frac{2}{\sigma ^3}\lambda\left( 1-\frac{1}{\mu}\right)\int_{\beta}^{1/2-\beta}g(t)\sigma ^2v(t)^2dt  \geq \frac{2   \Tilde{C}}{\sigma ^3}\lambda\left( 1-\frac{1}{\mu}\right)r^2 \phi(\beta)^2.\nonumber
\end{align}
Whence, the above two inequalities \eqref{ineq1H3} and \eqref{ineq2H3},  together with   \eqref{estimare mu,lambda}, yield
\begin{equation*}
      h'(\sigma )\leq \frac{2\Tilde{B}}{\sigma ^3} f(r\phi(\beta))r\phi(\beta)-\frac{2\Tilde{C}}{\sigma ^3}\lambda\left( 1-\frac{1}{\mu}\right)r^2 \phi(\beta)^2<0,
\end{equation*} 
as expected.

\textit{(c)} Assume that (H4) holds, and let $\sigma\in [r, R]$. Given this assumption and the symmetry of \( v \), we obtain
\[
\widetilde{h}(\sigma) = \sigma - 2 \int_{\beta}^{1/2} f(\sigma v(t)) g(t) v(t) \, dt.
\]
Since \( \sigma v(t) \in [r \phi(\beta), R] \) for all \( t \in [\beta, 1/2] \) by \eqref{inegalitate phi}, it follows that \( \widetilde{h} \) is of class \( C^2[r, R] \).
From (H4), one has
\[
\widetilde{h}''(\sigma) = -2 \int_\beta^{1/2} f''(\sigma v(t)) g(t) v(t)^3 \, dt \leq -2 \widetilde{C} \phi(\beta)^3 \min_{[r \phi(\beta), R]} f''(\cdot) < 0,
\]
which shows that \( \widetilde{h} \) is strictly concave on $[r,R]$.

Consequently, in all three cases, from \eqref{capete}, it follows immediately that $\alpha_u '(\tau)$ has exactly one zero within the interval  $\left(\frac{r}{|u|_{H_0  ^1}},\frac{R}{|u|_{H_0  ^1}} \right)$. Moreover, $\alpha_u '$ is positive on $\left( \frac{r}{|u|_{H_0  ^1}}, s(u)\right)$, while $\alpha'_u$ is negative on $\left(s(u), \frac{R}{|u|_{H_0  ^1}}\right)$, so  condition (h1) is verified.
\newline
To prove the last assertion (h4),  assume $u\in \mathcal{N}$. Thus, $r<|u|_{H_0 ^1}<R$ and
\begin{equation*}
    |u|_{H_0^1}^2=\int_0^1f(u(t))g(t)u(t)dt.
\end{equation*}
\textit{(a)'} 
In case that  (H2) holds, since $u(t)\geq \phi(\beta) |u|_{H_0 ^1}\geq r \phi(\beta)$ by \eqref{inegalitate phi}, we compute
 \begin{align*}
     E''(u)(u, u)&=|u|^2_{H_0^1}-\int_0^1 f'(u(t))g(t)u^2(t)dt
   \\&  =\int_0^1 f(u(t))g(t)u(t)dt-\int_0^1  f'(u(t))g(t)u(t)^2dt\\&
      \leq -2\int_0 ^{1/2}\theta(u(t)) g(t)u(t)dt
     \\& \leq -2r\phi(\beta) \int_{\beta}^{1/2}  \theta(u(t)) g(t)dt \\&\leq
-2\Tilde{C} r\phi(\beta) \min_{[r\phi(\beta),R]}\theta(\cdot) <0.
 \end{align*}
 \textit{(b)'} Assume (H3) is satisfied. Then, taking $\sigma =|u|_{H_0 ^1 }$ and performing the same computations as in the estimates for the derivative of $h$ in (b), we have
\begin{align*}
     E''(u)(u, u)&=2\int_0 ^{1/2}\left( f(u(t))- f'( u(t)) u(t)\right) g(t)u(t) dt
    \\& 
   \leq 2\int_0^{\beta} f(u(t))g(t)u(t)dt+2\int_{\beta}^{1/2} \left( f(u(t))- f'( u(t)) u(t)\right) g(t)u(t) dt\\&
   \leq 2\int_0^{\beta} f(u(t))g(t)u(t)dt-2\left( 1-\frac{1}{\mu}\right)\int_{\beta}^{1/2}  f'( u(t))  g(t)u(t)^2 dt\\&
    \leq 2r\phi(\beta) \left(\Tilde{B}f(r\phi(\beta))-\lambda \Tilde{C}\left(1-\frac{1}{\mu} \right){r\phi(\beta)} \right)<0,
\end{align*}
where the latter inequality follows from \eqref{estimare mu,lambda}.
\newline
\textit{(c)'} For the last case, using \eqref{g=0} from (H4) and  the symmetry of $v$, we see that  
\begin{align*}
       E''(u)(u, u)&=|u|^2_{H_0^1}-2\int_{\beta}^{1/2} f'(u(t))g(t)u(t)^2dt\\&=2\int_{\beta}^{1/2} f(u(t))g(t)u(t)dt-2\int_{\beta}^{1/2} f'(u(t))g(t)u(t)^2dt.
\end{align*}
Since $u(t)\in [r\phi(\beta), R]$ for all $t\in [\beta, 1/2]$, the conclusion follows immediately by an argument similar to that in (a)' , that is, 
\begin{equation*}
     E''(u)(u, u)\leq -2\Tilde{C} r\phi(\beta) \min_{[r\phi(\beta),R]}\Tilde{\theta}(\cdot) <0.
\end{equation*}

Therefore, in all three cases, condition (h4) is verified.
Summing up, we have the following result.

\begin{theorem}\label{teorema aplicatie}
    Assume (H1) and either (H2), (H3) or (H4) hold true. Then,   problem \eqref{ec aplicatie} has a solution $u^\ast\in K$ such that $r<|u^{\ast}|_{H_0 ^1}<R$.
\end{theorem}
\begin{proof}
    As established  above, assumptions (h1)-(h4) are verified, thus Theorem \ref{th principala} guarantees the existence of a sequence $u_n\in \mathcal{N}$ such that \begin{equation*}
        E(u_n)\to \inf_{\mathcal{N}}E \quad \text{ and }\quad E'(u_n)\to 0.
    \end{equation*}
Since \( L^2(0,1) \)  embeds compactly into \( H^{-1}(0,1) \), the operator \( N \) is completely continuous from \( H_0^1(0,1) \) into itself (see, e.g., \cite{p4}); therefore  Theorem \ref{corolar} applies and gives the conclusion. 
\end{proof}

It is worth providing some commentaries on the  conditions (H2), (H3), and (H4). 

\begin{description}
    \item[(i)] The first condition, (H2), naturally extends the standard condition obtained via the Nehari manifold method over the entire domain. For instance, let \( a > 0 \), \( p > 1 \), $g\equiv 1$ and choose \( 0 < r < R < \infty \) such that
\begin{equation}\label{cond exemplu}
    a < \frac{\pi}{r^{p}} \quad \text{and} \quad a > \frac{1}{R^{p-1} \beta^{p+1} (1 - 2\beta)^{p+2} },
\end{equation}
for some \( \beta \in (0, 1/4) \). Then, the function \( f(t) = a t^{p} \) satisfies (H1) and (H2). Indeed, conditions \eqref{cond exemplu} ensure (H1), while (H2) follows directly since
\[
t f'(t) - f(t) = a (p - 1) t^{p}.
\]
However, it is not difficult to see that this setup does not lead to multiplicity. To see this, suppose (H2) holds for two pairs of points \( (r_1, R_1) \) and \( (r_2, R_2) \) with \( r_1 < R_1 < r_2 < R_2 \). Then, as the map $f(t)/t$ is increasing on $[0,r_2]$ (see (a)), we have
\[
\frac{f(R_1)}{R_1} < \frac{f(r_2)}{r_2} < \frac{\pi}{\tilde{A}},
\]
which implies \( \alpha_u\left( \frac{R_1}{|u|_{H_0^1}} \right) > 0 \) for all \( u \in K \setminus \{0\} \), contrary to \eqref{capete}.
\item[(ii)]
Conditions (H3) and (H4) may lead to multiplicity, as they impose restrictions intrinsically linked to the values of \(r\) and \(R\). Both conditions require an Ambrosetti-Rabinowitz-type assumption, which is often encountered in the study of the existence of solutions for nonlinear equations.  

Additionally, (H3) includes an extra requirement on the derivative of $f$ and the associated coefficients, while (H4) assumes that the function \(g\) vanishes on a small interval starting from 0 and that the function \(f\) is strictly convex on the interval \([r\phi(\beta), R]\).  
We also remark that the three conditions (H2), (H3), and (H4) can be applied independently  to different pairs \((r_i, R_i)\). However, it is important to note that if (H2) is used, it should only be applied to the first pair \((r_1, R_1)\), as remark (i) shows.
\end{description}

\section*{Acknowledgements}
The author expresses gratitude to Prof. Radu Precup for his valuable insights on the subject and thorough verification of the paper.


\begin{thebibliography}{9}
\bibitem{p3} 
Agarwal, R., Meehan, M., O’Regan, D., Precup, R.: Location of nonnegative solutions for differential equations on finite and semi-infinite intervals. Dynam. Systems Appl. \textbf{12}(3-4), 323–331 (2003)



\bibitem{ambroseti} 
Ambrosetti, A., Malchiodi, A.: Nonlinear Analysis and Semilinear Elliptic Problems. Cambridge University Press (2007)



\bibitem{bahri} 
Bahri, A., Berestycki, H.: A perturbation method in critical point theory and applications. Trans. Am. Math. Soc. \textbf{267}(1), 1-32 (1981)


\bibitem{brezis} 
Brezis, H: Functional Analysis, Sobolev Spaces and Partial Differential Equations Springer, New York (2011)

\bibitem{green} 
Cabada, A.: Green’s Functions in the Theory of Ordinary Differential Equations. Springer, NY (2014)


\bibitem{pohozaev}Drabek, P., Pohozaev, S.I.: Positive solutions for the p-Laplacian: application of the fibering
method. Proc. Royal Soc. Edinb. A \textbf{127}(1), 703-726 (1997). 

\bibitem{ekeland} 
 Ekeland, I.: On the variational principle. J. Math. Anal. Appl. \textbf{47}(2), 324-353 (1974)


\bibitem{f} 
Figueiredo, D. G.: Lectures on the Ekeland Variational Principle with Applications and Detours. Tata Institute of Fundamental Research, Bombay (1989)


\bibitem{figueredo} 
Figueiredo, G.M., Ramos Quoirin, H., Silva, K.: Ground states of elliptic problems over cones. 	Calc. Var. Partial Differ. Equ. \textbf{60}(1), 189 (2021)


\bibitem{galewski} 
Galewski, M.: Localization properties for nonlinear equations involving monotone operators. Math. Meth. Appl. Sci. \textbf{43}(1), 9776–9789 (2020)


\bibitem{precup} 
Lisei, H., Precup, R., Varga, C.: A Schechter type critical point result in annular conical domains of a Banach space and applications. Discrete Contin. Dyn. Syst. \textbf{36}(7), 3775–3789 (2016) 

\bibitem{gateaux} 
McLaughlin, D.P., Poliquin, R.A., Vanderwerff, J.D., Zizler, V. E.: Second-order Gâteaux differentiable bump functions and approximations in Banach spaces. Can. J. Math. \textbf{45}(3), 612-625 (1993)



\bibitem{nehari original1} 
Nehari, Z.: On a class of nonlinear second-order differential equations. Trans. Amer. Math. Soc. \textbf{95}(1),  101–123 (1960)

\bibitem{nehari original2} 
Nehari, Z.: Characteristic values associated with a class of nonlinear second-order differential equations. Acta Math. \textbf{105}(1), 141–175 (1961)



\bibitem{nehari} 
Szulkin, A., Weth, T.: The method of Nehari manifold. Handbook of nonconvex analysis and applications, pp. 597-632. Int. Press, Somerville, MA (2010) 




\bibitem{p2} 
Precup, R.: Critical point localization theorems via Ekeland’s variational principle. Dynam. Systems Appl. \textbf{22}(1), 355-370 (2013)
\bibitem{p1} 
Precup. R.: On a bounded critical point theorem of Schechter.  Stud. Univ. Babeş-Bolyai Math. \textbf{58}(1), 87-95 (2013)


\bibitem{p4} 
Precup, R.: Linear and Semilinear Partial Differential Equations: An Introduction. Berlin, Boston: De Gruyter (2013)


\bibitem{precup2} 
Precup, R., Pucci, P., Varga, C.: A three critical points result in a bounded domain of a Banach space and applications. Differ. Integral Equ. \textbf{30}(1), 555–568 (2017)


\bibitem{schecter} 
Schechter, M.: Linking Methods in Critical Point Theory. Birkhäuser, Basel (1999)



\bibitem{minimax} 
 Willem, M.: Minimax Theorems. Birkhäuser, Boston (1996)

\end{thebibliography}
\end{document}